\documentclass[reqno]{amsart}
\usepackage[english]{babel}
\usepackage{amssymb,verbatim}
\usepackage[T1]{fontenc}
\newtheorem{theorem}{Theorem}[section]

\newtheorem{lemma}[theorem]{Lemma}
\newtheorem{corollary}[theorem]{Corollary}
\theoremstyle{definition}
\newtheorem{definition}[theorem]{Definition}

\theoremstyle{remark}
\newtheorem*{remark}{Remark}
\numberwithin{equation}{section}

\newcommand{\RE}{\mbox{$\mathbb{R}$}}
\newcommand{\C}{\mbox{$\mathbb{C}$}}

\begin{document}

\title{Extension and approximation of $m$-subharmonic functions}

\author{Per \AA hag}\address{Department of Mathematics and Mathematical Statistics\\ Ume\aa \ University\\SE-901~87 Ume\aa, Sweden \\ Sweden}\email{Per.Ahag@math.umu.se}

\author{Rafa\l\ Czy{\.z}}\address{Faculty of Mathematics and Computer Science, Jagiellonian University, \L ojasiewicza~6, 30-348 Krak\'ow, Poland}
\thanks{The second-named author was partially supported by NCN grant DEC-2013/08/A/ST1/00312.}\email{Rafal.Czyz@im.uj.edu.pl}

\author{Lisa Hed}\address{Department of Mathematics and Mathematical Statistics\\ Ume\aa \ University\\SE-901~87 Ume\aa, Sweden \\ Sweden}\email{Lisa.Hed@math.umu.se}

\keywords{exhaustion function, $m$-subharmonic function, Jensen measure, approximation, Dirichlet problem}
\subjclass[2010]{Primary 32U05, 32F17; Secondary 46A55.}

\begin{abstract}
Let $\Omega\subset \mathbb C^n$ be a bounded domain, and let $f$ be a real-valued function defined on the whole
topological boundary $\partial \Omega$. The aim of this paper is to find a characterization of the functions $f$ which can be extended to
the inside to a $m$-subharmonic function under suitable assumptions on $\Omega$. We shall do so by using a function
algebraic approach with focus on $m$-subharmonic functions defined on compact sets. We end this note
with some remarks on approximation of $m$-subharmonic functions.
\end{abstract}

\maketitle

\section{Introduction}

 In potential theory the notion of subharmonic functions, $\mathcal{SH}$, is of fundamental importance,  and in pluripotential theory the notion of plurisubharmonic functions, $\mathcal{PSH}$, is of the same importance. In 1985,
 Caffarelli et al.~\cite{CHS} proposed a model that makes it possible to study the common properties of potential and pluripotential theories, as well as the transition between them. It also gives a splendid tool in geometric constructions.  The core focus of the Caffarelli-Nirenberg-Spruck framework is what is known today as $m$-subharmonic functions, $\mathcal{SH}_m$. These functions are considered in
 complex space, and on different types of complex manifolds. If $n$ is the underlying dimension, then it holds that
\[
\mathcal{PSH}=\mathcal{SH}_n \subset\cdots\subset \mathcal{SH}_1=\mathcal{SH}\, .
\]
To mention a few references related to the Caffarelli-Nirenberg-Spruck model~\cite{Blocki_weak, huang, Li, phong, WanWang}.

Let $\Omega\subset \mathbb C^n$ be a bounded domain, and let $f$ be a real-valued function defined on the topological boundary $\partial \Omega$. It is well-known that one can not always extend $f$ to the inside to a $m$-subharmonic function. This is not possible even in the cases $m=1$, and $m=n$. The aim of this paper is to find a characterization of the functions $f$ that have this classical extension property, but in the process we shall also be interested in when this extension can be approximated in neighborhoods of $\bar\Omega$. The first obstruction is that $\Omega$ is only
assumed to be a bounded domain. This does not yield a satisfying amount of $m$-subharmonic functions. Therefore, we assume
that there exists at least one non-constant and negative $m$-subharmonic function $\psi:\bar{\Omega}\to\RE$ such that for any $c\in\RE$ the set $\{x\in\Omega:\psi(x)<c\}$ is relatively compact in $\Omega$ (see Definition~\ref{def_msubkomp} for the meaning of being $m$-subharmonic on $\bar\Omega$). A bounded domain in $\C^n$ that satisfies this condition is called $P_m$-hyperconvex. More about this in Section~\ref{sec Pmhxdomains}.

 Inspired by the work of Poletsky~\cite{Po,PO2}, and Poletsky and Sigurdsson~\cite{PS}, we use ideas from the theory of  function algebras defined on a compact set. In the mentioned references, the authors use the beautiful and
 intricate holomorphic disk-theory. Within the Caffarelli-Nirenberg-Spruck framework there are no Poletsy disks except in  the case $m=n$. Therefore, we uses the idea of duality between functions and its corresponding Jensen measures. In Section~\ref{sec Jensen and msh}, we introduce and study necessary properties of $m$-subharmonic functions defined on a compact set in $\C^{n}$, and with the help of those results we arrive in Section~\ref{sec_extention} at the following theorem:

 \bigskip

 \noindent \textbf{Theorem~\ref{ext_in_pm_hyp}.} \emph{Let $\Omega$ be a bounded $P_m$-hyperconvex domain in $\mathbb C^n$,  $1\leq m\leq n$, and let $f$ be a real-valued function defined on $\partial \Omega$. Then the following are equivalent:}

\medskip

\begin{enumerate}\itemsep2mm

\item \emph{there exists $F\in \mathcal{SH}_m(\bar \Omega)$ such that $F=f$ on $\partial \Omega$;}

\item $f\in \mathcal{SH}_m(\partial \Omega)$.
\end{enumerate}

\medskip

\noindent\emph{Furthermore, if $f$ is continuous on $\partial \Omega$, then the function $F$ can be chosen to be continuous on $\bar{\Omega}$.}

\bigskip

Theorem~\ref{ext_in_pm_hyp} is the first result of this kind within the Caffarelli-Nirenberg-Spruck model. It should be emphasized that
this is not only the classical Dirichlet problem that in the case $m=1$, can be traced back to the work of Brelot, Lebesgue, Perron, Poincar\'{e}, Wiener, and others. This since $F\in \mathcal{SH}_m(\bar \Omega)$, and by Theorem~\ref{cont}, these functions can be characterize by approximation
on neighborhoods of $\bar \Omega$. If $m=n$, then Theorem~\ref{ext_in_pm_hyp} was proved in~\cite{HP}.

A natural question that arises from Theorem~\ref{ext_in_pm_hyp} is how to decided wether a function $u$ is in  $\mathcal{SH}_m(\bar \Omega)$ or not. From Theorem~\ref{thm_pmboundary} it follows that under the assumption that $\Omega\subset\C^n$ is a bounded open set, and that $u\in\mathcal{SH}_m(\bar \Omega)$, then $u\in \mathcal{SH}_m(\Omega)$, and $u\in\mathcal{SH}_m(\partial\Omega)$. The converse statement is not always true. But if we assume that $\Omega$ is $P_m$-hyperconvex, then we prove in Theorem~\ref{cor_msubbdvalue} that
\[
u\in\mathcal{SH}_m(\bar \Omega) \quad\Leftrightarrow\quad u\in\mathcal{SH}_m(\Omega) \text{ and } u\in\mathcal{SH}_m(\partial\Omega)\, .
\]
This justify further the study of the geometry of domains that admits a negative exhaustion function that belongs to $\mathcal{SH}_m(\bar \Omega)$. This is done in
Section~\ref{sec Pmhxdomains}. We end this note
with some concluding remarks on uniform approximation of $m$-subharmonic functions (Section~\ref{sec approx}).

\bigskip

Background information on potential theory can be found in~\cite{armitage,doob,landkof}, and for more information about pluripotential theory in~\cite{demailly_bok,K}. A beautiful treatise on subharmonic and plurisubharmonic functions is the monograph~\cite{hormander} written by H\"ormander. Definition and basic properties of $m$-subharmonic functions can be found in~\cite{SA}.

\bigskip

One concluding remark is in order. There are well-developed axiomatic, and algebraic, potential theories that could
have been deployed in connection with this paper. We have chosen not to do so, and leave it for others to
draw full benefits of these abstract models in order to learn more about the Caffarelli-Nirenberg-Spruck framework on compact sets. We want to mention the references~\cite{ArsoveLeutwiler, BlietnerHansen3, Constantinescu, Gamelin}.

\section{Jensen measures and $m$-subharmonic functions}\label{sec Jensen and msh}

In this section we will define the class of $m$-subharmonic function defined on a compact set, $X\subset \mathbb C^n$, and we will prove some properties of such functions. Among other things, we shall show that these functions are closely connected to approximation by $m$-subharmonic functions defined on strictly larger domains. But, first we need some notions and definitions. Let $\mathcal{SH}_m^o(X)$ denote the set of functions that are the restriction to $X$ of functions that are $m$-subharmonic and continuous on some neighborhood of $X$. Furthermore, let  $\mathcal{USC}(X)$ be the set of upper semicontinuous functions defined on $X$. Next, we define a class of Jensen measures.

\begin{definition}\label{def_JzmK} Let $X$ be a compact set in $\C^n$, $1\leq m\leq n$, and let $\mu$ be a non-negative regular Borel measure defined on $X$ with $\mu(X)=1$. We say that $\mu$ is a \emph{Jensen measure with barycenter} $z\in X$ \emph{w.r.t.} $\mathcal{SH}_m^o(X)$ if
  \[
  u(z)\leq \int_{X} u \, d\mu \qquad\qquad \text{for all } u  \in \mathcal{SH}_m^o(X)\, .
  \]
The set of such measures will be denoted by $\mathcal{J}_z^m(X)$.
\end{definition}

\begin{remark}
If $X_1\subset X_2$ are compact sets in $\mathbb C^n$, then for every $z\in X_1$ it holds
\[
\mathcal{J}_z^m(X_1)\subset \mathcal{J}_z^m(X_2)\, .
\]
\end{remark}

We shall need the following convergence result in $\mathcal{J}_z^m(X)$. It is obtained in a standard way using the Banach-Alaoglu theorem, and therefore the proof is omitted.

\begin{theorem} \label{thm_convmeasures}
  Let $X$ be a compact set in $\C^n$. Let $\{z_k\} \subset X$ be a sequence that is converging to $z$, as $k\to\infty$. For each $k$, let $\mu_k \in \mathcal{J}_{z_k}^m(X)$. Then there is a subsequence $\{\mu_{k_j}\}$, and a measure $\mu \in \mathcal{J}_z^m(X)$ such that $\mu_{k_j}$ converges weak-$^\ast$ to $\mu$.
\end{theorem}

Using the Jensen measures in Definition~\ref{def_JzmK} we shall now define what it means for a function to be $m$-subharmonic on a compact set.

\begin{definition} \label{def_msubkomp}
   Let $X$ be a compact set in $\C^n$. An upper semicontinuous function $u$ defined on $X$ is said to be \emph{$m$-subharmonic on $X$}, $1\leq m\leq n$,  if
\[
u(z) \leq \int_X u \, d\mu\, , \  \text { for all } \ z \in X \  \text { and all }\  \mu \in \mathcal{J}_z^m(X)\, .
\]
The set of  $m$-subharmonic defined on $X$ will be denoted by $\mathcal{SH}_m(X)$.

\end{definition}

\begin{remark}
By definition, we see that $\mathcal{SH}_m^o(X) \subset \mathcal{SH}_m(X)$.
\end{remark}

It is easy to see that $m$-subharmonic functions on compact sets share a lot of basic properties with $m$-subharmonic functions on open sets. Some of these properties are listed below.

\begin{theorem} \label{thm_basicprop2}
Let $X$ be a compact set in $\mathbb C^n$, and $1\leq m\leq n$. Then

\medskip

    \begin{enumerate}\itemsep2mm
    \item if $u,v \in \mathcal{SH}_m(X)$, then $su+kv \in \mathcal{SH}_m(X)$ for $s,k\geq 0$;

    \item if $u,v \in \mathcal{SH}_m(X)$, then $\max\{u,v\}\in \mathcal{SH}_m(X)$;

    \item if $u_j\in \mathcal{SH}_m(X)$ is a decreasing sequence, then $u=\lim_{j\to \infty} u_j\in \mathcal{SH}_m(X)$, provided $u(z)>-\infty$ for some point $z\in X$;

    \item if $u\in \mathcal{SH}_m(X)$ and $\gamma:\mathbb R\to\mathbb R$ is a convex and nondecreasing function,  then $\gamma\circ u\in\mathcal{SH}_m(X)$.
        \end{enumerate}
\end{theorem}
\begin{proof}
Properties $(1)$ and $(2)$ follows by Definition \ref{def_msubkomp}. To prove $(3)$, let $u_j\searrow u$. Then we
have that $u\in \mathcal{USC}(\bar\Omega)$. For $z\in X$, $\mu\in \mathcal J_z^m(X)$, we have by the monotone convergence theorem that
\[
u(z)=\lim_{j\to \infty}u_j(z)\leq \lim_{j\to \infty}\int u_j\,d\mu=\int\lim_{j\to \infty}d\mu=\int u\,d\mu\, .
\]
Part $(4)$ is a consequence of the Jensen inequality.
\end{proof}

The set $\mathcal{SH}_m^o(X)$ is a convex cone of continuous functions containing the constants, and separating points, and therefore we can apply the techniques of Choquet theory to get the following two versions of  Edwards' duality theorem.
Generalizations of Edwards' Theorem can be found in~\cite{GogusPerkinsPoletsky}.

\begin{theorem}\label{thm_edwards}
Let $X$ be a compact subset in $\mathbb C^n$, $1\leq m\leq n$, and let $\phi$ be a real-valued lower semicontinuous function defined
on $X$. Then we have
\begin{enumerate}
 \item[(a)]
\[
\sup\left \{\psi(z):\psi \in \mathcal{SH}_m^o(X) , \psi \leq \phi \right\}=\inf\left \{\int \phi \, d\mu : \mu \in \mathcal{J}_z^m(X)\right\}\, ,
\ \text {and}
\]
\item[(b)]
\begin{multline*}
\sup\left\{\psi (z) : \psi \in \mathcal{SH}_m(X)\cap \mathcal{C}(X), \psi \leq \phi\right\}\\=\sup\left\{\psi (z) : \psi \in \mathcal{SH}_m(X), \psi \leq \phi\right\}
=\inf\left\{\int \phi \, d\mu: \mu \in \mathcal{J}_z^m(X)\right\}\, .
\end{multline*}
\end{enumerate}
\end{theorem}
\begin{proof}
Part (a) is the direct consequence of Edwards' Theorem, and the proof of part (b) is postponed until after Theorem~\ref{cont} is proved.
\end{proof}

One important reason to study $m$-subharmonic functions on compact sets is that they are connected to approximation. In the case $m=1$, Theorem~\ref{cont} goes back to Debiard and Gaveau~\cite{DebiardGaveau}, and  Bliedtner and Hansen~\cite{BlietnerHansen1,BlietnerHansen2}(see also~\cite{perkins,perkins2}). In the case $m=n$, part $(a)$, was shown by Poletsky in \cite{Po}, and part $(b)$ in~\cite{CHP} . In
Section~\ref{sec approx}, we shall have some concluding remarks in connection with this type of approximation.

\begin{theorem}\label{cont} Let $X \subset \C^n$ be a compact set, and $1 \leq m \leq n$.
\begin{itemize}\itemsep2mm
\item[$(a)$] Let $u\in \mathcal {USC}(X)$. Then $u \in \mathcal{SH}_m(X)\cap \mathcal{C}(X)$ if, and only if, there is a sequence $u_j \in \mathcal{SH}_m^o(X)$ such that $u_j \nearrow u$ on $X$.

\item[$(b)$]  Then $u \in \mathcal{SH}_m(X)$ if, and only if, there is a sequence $u_j \in \mathcal{SH}_m^o(X)$ such that $u_j \searrow u$.
\end{itemize}
\end{theorem}
\begin{proof} \emph{Part $(a)$:} Let $u \in \mathcal{SH}_m(X)\cap \mathcal{C}(X)$. Since the Dirac measure $\delta_z$ is in  $\mathcal{J}_z^m(X)$,
we have that
\[
u(z)=\inf\left\{\int u \, d\mu: \mu \in \mathcal{J}_z^m(X)\right\}\, .
\]
Theorem \ref{thm_edwards} part (a), yields now that
\[
u(z)=\inf\left \{\int u \, d\mu: \mu \in \mathcal{J}_z^m(X)\right\}=\sup\left\{\varphi(z): \varphi \in \mathcal{SH}_m^o(X), \varphi \leq u\right\}\, .
\]
Since the functions in $\mathcal{SH}_m^o(X)$ are continuous, Choquet's lemma (see e.g. Lemma~2.3.4 in~\cite{K}) says that
there exists a sequence $u_j \in \mathcal{SH}_m^o(X)$ such that $u_j \nearrow u$.

Now assume that there exists a sequence
$u_j \in \mathcal{SH}_m^o(X)$ such that $u_j \nearrow u$. Then $u$ can be written as the supremum of continuous functions.
Hence, $u$ is lower semicontinuous. Thus, $u$ is continuous. Let $z \in X$, and $\mu \in \mathcal{J}_z^m(X)$, then
\[
u(z)=\lim_j u_j(z)\leq \lim_j \int u_j \, d\mu =\int \lim_j u_j \, d\mu =\int u \, d\mu\, .
\]
By Definition~\ref{def_msubkomp} we know that $u \in \mathcal{SH}_m(X)$.

\bigskip

\emph{Part $(b)$:} First assume that $u$ is the pointwise limit of a decreasing sequence of functions $u_j \in \mathcal{SH}_m^o(X)$.
Then we have that $u\in \mathcal {USC}(X)$. Let $z \in X$ and $\mu \in \mathcal{J}_z^m(X)$, then it follows that
\[
u(z)=\lim_j u_j(z)\leq \lim_j \int u_j \, d\mu= \int\lim_j\,d\mu=\int u \, d\mu\, .
\]
Hence $u \in \mathcal{SH}_m(X)$.

For the converse, assume that $u \in \mathcal{SH}_m(X)$. We now want to show that there is a sequence of
functions $u_j \in \mathcal{SH}_m^o(X)$ such that $u_j \searrow u$ on $X$. We begin by showing that for every $f \in \mathcal{C}(X)$ with $u < f$
on $X$, we can find $v \in  \mathcal{SH}_m^o(X)$ such that $u < v \leq f$. Let
\[
F(z)=\sup\{\varphi(z): \varphi \in \mathcal{SH}_m^o(X), \varphi \leq f\} \, .
\]
From Theorem \ref{thm_edwards} part (a) it follows now that
\[
  F(z)=\inf\left\{\int f \, d\mu : \mu \in \mathcal{J}_z^m(X)\right\}\, .
\]
From the Banach-Alaoglu theorem we know that $\mathcal{J}_z^m(X)$ is weak-$^\ast$ compact, hence for all $z \in X$ we can find
$\mu_z \in \mathcal{J}_z^m(X)$ such that
\[
F(z)=\int f \, d\mu_z\, .
\]
We have
\[
F(z)=\int f \, d\mu_z > \int u \, d\mu_z \geq u(z)\, .
\]
Hence, $u < F$. By the construction of $F$ we know that for every given $z \in X$, there exists a function $v_z \in \mathcal{SH}_m^o(X)$ such
that $v_z \leq F$ and $u(z)<v(z) \leq F(z).$ Since the function $u-v_z$ is upper semicontinuous, then the set
\[
U_z=\{y\in X: u(y)-v_z(y)<0\}
\]
is open in $X$. By assumption $X$ is compact, and therefore there are finitely many points $z_1,\ldots,z_k$ with corresponding functions
$v_{z_1},\ldots,v_{z_k}$, and open sets $U_{z_1},\ldots, U_{z_k}$, such that $u < v_{z_j}$ on $U_{z_j}$. Furthermore,
\[
X=\bigcup_{j=1}^k U_{z_j}\, .
\]
The function $v=\max\{v_{z_1},\ldots,v_{z_k}\}$ is in $\mathcal{SH}_m^o(X)$, and $u < v \leq f$. We are now ready to prove that $u$ can be
approximated as in the statement in the theorem. The function $u$ is upper semicontinuous, and therefore it can be approximated with a decreasing
sequence $\{f_j\}$ of continuous functions. We can then find $v_1 \in \mathcal{SH}_m^o(X)$ such that $u < v_1 \leq f_1$. If we now assume that we
can find a decreasing sequence of functions $\{v_1,\ldots,v_k\}$ such that $v_j \in \mathcal{SH}_m^o(X)$, and $u < v_j$ for $j=1,\ldots,k$, then we can
find a function $v_{k+1} \in \mathcal{SH}_m^o(X)$ such that $u<v_{k+1}$ and $v_{k+1} \leq \min\{f,v_k\}$. Now the conclusion of the theorem follows by
induction.
\end{proof}

\begin{remark} In Theorem~\ref{cont} part $(a)$ we have uniform approximation on $X$. One can assume that the decreasing sequence in Theorem~\ref{cont}
part $(b)$ is smooth. This follows from a standard diagonalization argument.
\end{remark}

\begin{proof}[Proof of Theorem~\ref{thm_edwards} part (b).]
Let us define the following families of probability measures defined on $X$
\[
\begin{aligned}
&\mathcal{M}_z^m(X)=\left \{\mu: \, u(z)\leq \int u\, d\mu\, , \,\,  \forall u\in \mathcal{SH}_m(X)\cap \mathcal C(X)\right \}, \\
&\mathcal{N}_z^m(X)=\left \{\mu: \, u(z)\leq \int u\, d\mu\, , \,\, \forall u\in \mathcal{SH}_m(X) \right \}\, . \\
\end{aligned}
\]
We have
\[
\mathcal{N}_z^m(X)\subset \mathcal{M}_z^m(X)\subset \mathcal{J}_z^m(X)\, ,
\]
since $\mathcal{SH}_m^o(X)\subset\mathcal{SH}_m(X)\cap \mathcal C(X)\subset \mathcal{SH}_m(X)$. On the other hand, let $z\in X$, $\mu \in \mathcal{J}_z^m(X)$,  and let $\varphi\in \mathcal{SH}_m(X)$, then by Theorem~\ref{cont} part (b) there exists a decreasing sequence $u_j\in \mathcal{SH}_m^o(X)$ such that $u_j\to \varphi$, when $j\to \infty$, and then we have
\[
\varphi(z)=\lim_{j\to \infty}u_j(z)\leq \lim_{j\to \infty}\int u_j\,d\mu=\int\lim_{j\to \infty}u_j\,d\mu=\int\varphi\, d\mu\, .
\]
Hence, $\mu\in \mathcal{N}_m(X)$, and therefore $\mathcal{J}_z^m(X)\subset \mathcal{N}_z^m(X)$.
\end{proof}

A direct consequence of Theorem~\ref{cont} part $(b)$ is the following corollary.

\begin{corollary}\label{cor}
If $X_1\subset X_2$ are compact sets in $\mathbb C^n$, then $\mathcal{SH}_m(X_2)\subset \mathcal{SH}_m(X_1)$.
\end{corollary}
\begin{proof}
To see this take $u\in \mathcal{SH}_m(X_2)$, then by Theorem~\ref{cont} part $(b)$  there exists a sequence $u_j\in \mathcal{SH}_m^o(X_2)$ decreasing to $u$. Since $u_j$ belongs also to $\mathcal{SH}_m^o(X_1)$ then $u\in \mathcal{SH}_m(X_1)$.
\end{proof}

In Theorem~\ref{localization_Pm}, we shall need the following localization theorem. The case $m=n$ is Gauthier's localization theorem from~\cite{Gauthier}. For the proof of the following theorem, and later sections we need to recall the following definition. A function
$u$ is said to be \emph{strictly $m$-subharmonic} on $\Omega$ if for every $p \in \Omega$ there exists a constant $c_p >0$ such that $u(z)-c_p |z| ^2$ is $m$-subharmonic in a neighborhood of $p$.

\begin{theorem}\label{localization}
If $X \subset \C^n$ is a compact set, then $u \in \mathcal{SH}_m(X)\cap \mathcal{C}(X)$ if, and only if, for each $z \in X$, there is a neighborhood
$B_z$ such that $u |_{X\cap \bar B_z}\in \mathcal{SH}_m(X\cap \bar B_z)\cap \mathcal{C}(X\cap \bar B_z)$.
\end{theorem}
\begin{proof} This proof is inspired by~\cite{Gauthier}. First we see that the restriction of a function $u \in \mathcal{SH}_m(X)\cap \mathcal C(X)$ to $X \cap \bar B$ is $m$-subharmonic on that set. This follows from Corollary~\ref{cor}. Now we show the converse statement. Since $X$ is compact there exists a finite open covering $\{B_j\}$ of $X$. Assume that $u|_{X \cap \bar B_j} \in \mathcal{SH}_m(X \cap \bar B_j)\cap \mathcal C(X \cap \bar B_j)$ for all $j$. For every $j$, we can find compact sets $K_{j,k}$ such that $K_{j,k} \subset B_k$ and
\[
\partial B_j \cap X \subset \bigcup_{k \neq j}K_{j,k}\, .
\]
Let $K_k = \bigcup_j K_{j,k}$,  and note that $K_k \subset B_k$. Set
\[
d_k = \mbox{dist}(K_k, \partial B_k)\, .
\]
For every $k$ there exists a function $\chi_k$ that is smooth on $\C^n$, $-1 \leq \chi_k \leq 0$, $\chi_k(z)=0$ when $\mbox{dist}(z, K_k) \leq \frac{d_k}{2}$, and $\chi_k=-1$ outside of $B_k$. Choose an arbitrary constant $c > 0$. The function $|z|^2$ is strictly $m$-subharmonic, so there exists a constant $\eta^0_k>0$ such that for every $0 < \eta_k<\eta_k^0$, the function $\eta_k \chi_k +c|z|^2$ is $m$-subharmonic and continuous on an open set $V_k$, $B_k \Subset V_k$. Choose a sequence $\{\varepsilon_j\}$ of positive numbers such that
\begin{equation}\label{eq:etaepsilon}
2 \max_{z \in \bar B_j} \varepsilon_j < \min_{z \in \bar B_j} \eta_j\, ,
\end{equation}
for every $z \in X$. The reason for this will be clear later. By the  assumption that $u|_{X \cap \bar B_j} \in \mathcal{SH}_m(X \cap \bar B_j)\cap \mathcal C(X \cap \bar B_j)$ for every $j$, Theorem~\ref{cont} part (b) says that there exist open sets  $U_j$, $(X \cap B_j) \Subset U_j \Subset V_j$ and functions $u_j \in \mathcal{SH}_m(U_j)\cap \mathcal C(U_j)$ such that
\begin{equation} \label{eq:epsilon_j}
|u-u_j|<\varepsilon_j \ \text{on} \ X \cap \bar B_j\, .
\end{equation}
For $z \in (U_j \setminus X)\cup (X\cap \bar B_j)$ set
\[
f_j(z)=u_j(z) + \eta_j \chi_j(z) + c|z|^2\, ,
\]
and elsewhere set $f_j=-\infty.$ Now define the function
\[
v(z)=\max_j f_j(z)\, .
\]
It remains to show that $v$ approximates $u$ uniformly on $X$, and that $v \in \mathcal{SH}_m^o(X)$. For $z \in X$ we have
\begin{multline}\label{exp}
|u(z)-v(z)| = |u(z)-\max_j f_j(z)|\\
= \left|u(z)-\max_{z \in X\cap \bar B_j}(u_j(z)+\eta_j \chi_j + c|z|^2)\right|\\
\leq \max_{z \in X\cap \bar B_j} \eta_j + |u(z)-\max_{z \in X\cap \bar B_j} u_j(z)| + c|z|^2\, .
\end{multline}
By choosing the constants $c, \eta_j, \varepsilon_j$ in the right order and small enough, then the right-hand side of~(\ref{exp}) can be made arbitrary small. Hence, $v$ approximates $u$ uniformly on $X$.

To prove that $v \in \mathcal{SH}_m^o(X)$, first take $z \in X$ that does not lie on the boundary of any $B_j$. The functions $f_k$, that are not $- \infty$ at $z$, are finitely many and they are continuous and $m$-subharmonic in a neighborhood of $z$. If $z \in \partial B_j \cap X$, then there exists a $k$ such that $z \in (X \cap K_k) \subset (X \cap B_k)$. For this $j$ and $k$ we have
\begin{multline*}
f_j(z)=u_j(z)+\eta_j \chi_j + c|z|^2
= u_j(z) - \eta_j +c|z|^2\\
=\bigl(u_j(z)-u_k(z)\bigr)+\bigl(u_k(z)+\eta_k 0+c|z|^2\bigr) - \eta_j\\
=f_k(z)+\bigl(u_j(z)-u_k(z)\bigr)-\eta_j \leq f_k(z)
\end{multline*}
where the last inequality follows from assumption (\ref{eq:etaepsilon}) together with (\ref{eq:epsilon_j}) (that makes sure that $|u_j(z)-u_k(z)|<\varepsilon_j+ \varepsilon_k$). This means that locally, near $z$, we can assume that the function $v$ is the maximum of functions $f_k$, $k \neq j$, where the functions $f_k$ are continuous and $m$-subharmonic in a neighborhood of $z$. This concludes the proof.
\end{proof}

As an immediate consequence we get the following gluing theorem for $m$-sub\-har\-mo\-nic functions on compact sets.

\begin{corollary}
Let $\omega \Subset \Omega$ be open sets, let $u\in \mathcal{SH}_m(\bar \omega)\cap \mathcal C(\bar \omega)$, $v\in \mathcal{SH}_m(\bar\Omega)\cap\mathcal C(\bar\Omega)$ and $u(z)\leq v(z)$ for $z\in \partial \omega$. Then the function
\[
\varphi=\begin{cases}
v, \, \text { on } \, \bar\Omega\setminus \omega,\\
\max\{u,v\}, \; \text { on } \, \omega,
\end{cases}
\]
belongs to $\mathcal{SH}_m(\bar \Omega)\cap\mathcal C(\bar \Omega)$.
\end{corollary}

\begin{proof}
Let $\varepsilon >0$ and define
\[
\varphi_{\varepsilon}=\begin{cases}
v+\varepsilon, \, \text { on } \, \bar\Omega\setminus \omega,\\
\max\{u,v+\varepsilon\}, \; \text { on } \, \omega.
\end{cases}
\]
Then by Theorem~\ref{localization} we get that $\varphi_{\varepsilon}\in \mathcal{SH}_m(\bar \Omega)\cap\mathcal C(\bar \Omega)$ and $\varphi_{\varepsilon}\searrow \varphi$, as $\varepsilon \to 0$. By Theorem~\ref{thm_basicprop2} we conclude that $\varphi\in \mathcal{SH}_m(\bar \Omega)\cap\mathcal C(\bar \Omega)$.
\end{proof}

Let us now look at a bounded domain $\Omega$ in $\mathbb C^n$. We want to investigate what the connection is between $\mathcal{SH}_m(\bar \Omega)$ and $\mathcal{SH}_m(\Omega)$. It is easy to show that $\mathcal{SH}_m(\bar \Omega) \subset \mathcal{SH}_m(\Omega)$. Using Definition \ref{def_msubkomp} we know that a function $\varphi \in \mathcal{SH}_m(\Omega)\cap \mathcal{USC}(\bar \Omega)$ is in $\mathcal{SH}_m(\bar \Omega)$ if $\varphi(z) \leq \int \varphi \, d\mu$ for all $\mu \in  \mathcal{J}_z^m(\bar \Omega)$ where $z \in \bar \Omega$. In the same way as in \cite{HP} we can show that it is enough to look at the measures in $\mathcal{J}_z^m(\bar \Omega)$ for $z \in \partial \Omega$.

\begin{theorem} \label{thm_pmboundary}
  Let $\Omega$ be a bounded open set in $\C^n$, and $1 \leq m \leq n$.

\medskip

\begin{enumerate}\itemsep2mm
  \item[$(1)$] If $\varphi \in \mathcal{SH}_m(\bar \Omega)$,  then $\varphi \in \mathcal{SH}_m(\Omega)$ and $\varphi \in \mathcal{SH}_m(\partial\Omega)$.

  \item[$(2)$] If $\varphi \in \mathcal{SH}_m(\Omega)\cap \mathcal{USC}(\partial \Omega)$, and
  \[
  \varphi(z) \leq \int \varphi \, d\mu\, , \  \text { for all } \ z \in \partial\Omega \  \text { and all }\  \mu \in \mathcal{J}_z^m(\bar \Omega)\, ,
  \]
   then $\varphi \in \mathcal{SH}_m(\bar \Omega).$
\end{enumerate}
\end{theorem}
\begin{proof}
\emph{Part} $(1):$ By Theorem~\ref{cont} part $(b)$ there exists a sequence $\varphi_j\in \mathcal{SH}_m^o(\bar \Omega)$ decreasing to $\varphi$.
Then $\varphi_j\in \mathcal{SH}_m(\Omega)$, so $\varphi \in \mathcal{SH}_m(\Omega)$. The fact that
$\varphi \in \mathcal{SH}_m(\partial\Omega)$ follows from Corollary~\ref{cor}.

\bigskip

\emph{Part} $(2):$ By Theorem \ref{cont} part $(b)$ we want to prove that there is a decreasing sequence of functions
$\varphi_j$ in $\mathcal{SH}_m^o(\bar \Omega)$ such that $\varphi_j \rightarrow \varphi$ on $\bar \Omega$. Since $\varphi$
is upper semicontinuous we can find $\{u_j\} \subset \mathcal{C}(\bar \Omega)$ such that $u_j \searrow \varphi$ on $\bar \Omega$.
We are going to show that we can find functions $\{v_j\} \in \mathcal{SH}_m^o(\bar \Omega)$ such that $v_j \leq u_j$ and $v_j(z) \searrow \varphi(z)$
for every $z \in \partial \Omega$. From this it will follow that the functions
\[
\varphi_j=\begin{cases}
\max\{\varphi(z),v_j(z)\} & \text{if} \ z \in \bar \Omega\\
v_j(z) & \text{otherwise}\\
\end{cases}
\]
will belong to $\mathcal{SH}_m^o(\bar \Omega)$, and $\varphi_j \searrow \varphi$ on $\bar \Omega$.

To construct the approximating sequence $\{v_j\}$ define first
\[
F_j(z):=\sup\left \{v(z): v \in \mathcal{SH}_m^o(\bar \Omega), v \leq u_j\right \}=\inf\left \{\int u_j \, d\mu: \mu \in \mathcal{J}_z^m(\bar \Omega)\right\}.
\]
Since $\mathcal{J}_z^m(\bar \Omega)$ is compact in the weak$^\ast$-topology we can, for all $z\in \bar \Omega$ find $\mu_z \in \mathcal{J}_z^m(\bar \Omega)$ such that $F_j(z)=\int u_j \, d\mu_z$. We know, by the construction of $F_j$, that $F_j \leq u_j$, and
\[
F_j(z)=\int u_j \, d\mu_z > \int \varphi \, d\mu_z \geq \varphi(z) \ \text{ for all }z \in \partial \Omega\, .
\]
By the construction of $F_j$ we know that for every $z\in \partial \Omega$ we can find $v_z \in \mathcal{SH}_m^o(\bar \Omega)$ such that $v_z \leq F_j$ and $\varphi(z)< v_z(z)\leq F_j(z)$. The function $\varphi - v_z$ is upper semicontinuous and therefore the set
\[
U_z=\{w \in \partial \Omega: \varphi(w)-v_z(w)<0\}
\]
is open in $\partial \Omega$. It now follows from the compactness of $\partial \Omega$ that there are finitely many points $z_1,\ldots,z_k$ with corresponding functions $v_{z_1},\ldots,v_{z_k}$ and open sets $U_{z_1},\ldots,U_{z_k}$ such that $\varphi < v_{z_j}$ in $U_{z_j}$ and $\partial \Omega=\cup_{j=1}^kU_{z_j}$. The function $v_j=\max\{v_{z_1},\ldots,v_{z_k}\}$ belongs to $\mathcal{SH}_m^o(\bar \Omega)$ and $\varphi(z)< v_j(z) \leq u_j(z)$ for $z \in \partial \Omega$. This completes the proof.

\end{proof}

\section{$P_m$-hyperconvex domains}\label{sec Pmhxdomains}

Assume that $\Omega\subset\C^n$ is a bounded open set, and let  $1 \leq m \leq n$. Theorem~\ref{cont}
give rise to the question of how to decide if $u$ is in  $\mathcal{SH}_m(\bar \Omega)$.
From Theorem~\ref{thm_pmboundary} it follows that if $u\in\mathcal{SH}_m(\bar \Omega)$, then
$u\in \mathcal{SH}_m(\Omega)$, and $u\in\mathcal{SH}_m(\partial\Omega)$. The converse statement is not true, not even
under the assumption that $\Omega$ is $m$-hyperconvex (see Definition~\ref{def_mhx}). But if we assume that $\Omega$ admits a negative exhaustion function in $\mathcal{SH}_m(\bar \Omega)$ (notice here that $\bar\Omega$ is a compact set), then we shall prove
in Theorem~\ref{cor_msubbdvalue} that
\[
u\in\mathcal{SH}_m(\bar \Omega) \quad\Leftrightarrow\quad u\in\mathcal{SH}_m(\Omega) \text{ and } u\in\mathcal{SH}_m(\partial\Omega)\, .
\]

First we shall recall the definition of a $m$-hyperconvex domain.

\begin{definition}\label{def_mhx}
  Let $\Omega$ be a domain in $\C^n$, and $1 \leq m \leq n$. We say that $\Omega$ is \emph{$m$-hyperconvex} if it admits an exhaustion function that is negative and in $\mathcal{SH}_m(\Omega)$.
\end{definition}

Let us now make a formal definition of $P_m$-hyperconvex domains.

\begin{definition}\label{def_Pmhx}
  Let $\Omega$ be a domain in $\C^n$, and let $1 \leq m \leq n$. We say that $\Omega$ is \emph{$P_m$-hyperconvex} if it admits an exhaustion function that is negative, and in $\mathcal{SH}_m(\bar \Omega)$.
\end{definition}

From Theorem~\ref{thm_pmboundary} it follows that a $P_m$-hyperconvex domain is also $m$-hyper\-con\-vex. The
converse is not true. The case $m=n$ was studied in~\cite{HP}, and discussed in~\cite{PW}. A $P_n$-hyperconvex domain is $P_m$-hyperconvex for every $m=1,\ldots,n$, and as observed in~\cite{HP}, the notion of $P_n$-hyperconvexity is strictly weaker than the notion of \emph{strict hyperconvexity} that has been studied and used by
for example Bremermann~\cite{bremermann}, and Poletsky~\cite{PO3}. Furthermore, a $P_m$-hyperconvex domain is fat in the sense $\Omega=(\bar{\Omega})^{\circ}$.

 It is straight forward to see that if $\Omega_1$ and $\Omega_2$ are $P_m$-hyperconvex domains in $\C^n$, then
 $\Omega_1 \cap \Omega_2$ is $P_m$-hyperconvex in $\C^n$, and $\Omega_1 \times \Omega_2$ is $P_m$-hyperconvex in $\mathbb C^{2n}$.

\bigskip

 As in the case of $m$-hyperconvex domains, we have in Theorem~\ref{thm_charPmhx} several nice characterizations of $P_m$-hyperconvex domains in terms of the barrier functions, and Jensen measures. The property that a domain is (globally) $P_m$-hyperconvex if, and only if, it is locally $P_m$-hyperconvex we leave to Theorem~\ref{localization_Pm}.

\begin{theorem}\label{thm_charPmhx}
Let $\Omega$ be a bounded domain in $\C^n$. Then the following assertions are equivalent:

\medskip
\begin{enumerate}\itemsep2mm
\item $\Omega$ is $P_m$-hyperconvex in the sense of Definition~\ref{def_Pmhx};

\item $\Omega$ admits a negative exhaustion function that is in $\mathcal{SH}_m(\bar \Omega)\cap \mathcal{C}(\bar \Omega)$;

\item $\partial \Omega$  has a weak barrier at every point $z_0\in\partial \Omega$ that is in $\mathcal{SH}_m(\bar \Omega)$, i.e. there exists a function  $u\in\mathcal{SH}_m(\bar \Omega)$, such that $u<0$ on $\Omega$ and
     \[
     \lim_{z\to z_0\atop z\in\Omega} u(x)=0\, ;
     \]

\item for every $z\in \partial \Omega$, and every $\mu \in \mathcal{J}_z^m(\bar \Omega)$, we have that $\operatorname{supp} (\mu) \subseteq \partial \Omega$;

\item $\Omega$ admits a continuous negative exhaustion function which is $m$-subharmonic on $\bar \Omega$, smooth and strictly $m$-subharmonic on $\Omega$.

\end{enumerate}
\end{theorem}
\begin{proof}
  $(1) \Rightarrow (4):$ Assume that $\Omega$ is $P_m$-hyperconvex, then there exists a negative exhaustion function $\psi \in \mathcal{SH}_m(\bar \Omega)$. Take $z \in \partial \Omega$ and let $\mu \in \mathcal{J}_z^m(\bar \Omega)$, then
\[
0=\psi(z)\leq \int \psi \, d\mu \leq 0\, .
\]
  Since $\psi <0 $ on $\Omega$, we have that $\operatorname{supp}(\mu) \subseteq \partial \Omega$.

\medskip

$(2) \Rightarrow (1):$ Follows by Definition \ref{def_Pmhx}.

\medskip

For the implications $(4)\Rightarrow (3)$, $(4) \Rightarrow (2)$, and $(4)\Rightarrow (1)$, assume that for all $w \in \partial \Omega$, the every measures $\mu \in \mathcal{J}_w^m(\bar \Omega)$ satisfy $\operatorname{supp} (\mu) \subseteq \partial \Omega$. Let $z\in \Omega$, $r>0$ be such that $B(z,r)\Subset \Omega$ and let
\[
u(z)= \sup\{\varphi(z): \varphi \in \mathcal{SH}_m(\bar\Omega)\cap \mathcal{C}(\bar \Omega), \varphi \leq 0, \varphi \leq -1 \ \text{on} \ B(z,r)\}\, .
\]
Then $u$ is lower semicontinuous, and by Theorem~\ref{thm_edwards} part (b), we have that
\[
u(z)=\inf\left\{\int -\chi_{B(z,r)} \, d\mu : \mu \in \mathcal{J}_z^m(\Omega)\right\}=-\sup\left\{\mu(B(z,r)): \mu \in \mathcal{J}_z^m(\Omega)\right\}\, .
\]
We shall prove that $\lim_{\xi\to \partial \Omega}u(\xi)=0$. Assume the contrary, i.e. that there is a point $z \in \partial \Omega$ such that $\liminf_{\xi \rightarrow z}u(\xi)<0$. Then we can find a sequence $z_n \rightarrow z$ such that $u(z_n)<-\varepsilon$ for every $n$. We can find corresponding measures $\mu_n \in \mathcal{J}_{z_n}^m(\bar\Omega)$ such that $\mu_n(B(z,r))>\varepsilon$. By Theorem~\ref{thm_convmeasures} we can (by passing to a subsequence) assume that $\mu_n$ converges weak-$^\ast$ to a measure $\mu \in \mathcal{J}_z^m(\bar \Omega)$. Then, using Lemma~2.3 in \cite{CCW}, we have that
\[
\mu(\overline {B(z,r)})=\int \chi_{\overline {B(z,r)}} \, d\mu \geq \limsup_{n \rightarrow \infty} \int \chi_{\overline {B(z,r)}} \, d\mu_n =\limsup_{n \rightarrow \infty} \mu_n(\overline {B(z,r)}) >\varepsilon\geq 0.
\]
This contradicts the assumption that $\mu \in \mathcal{J}_z^m(\bar \Omega)$ only has support on the boundary. It remains to show that $u \in \mathcal{SH}_m(\bar \Omega)\cap \mathcal{C}(\bar \Omega)$.
We have that $u^*\in \mathcal {SH}_m(\Omega)\cap \mathcal {USC}(\bar \Omega)$ and $\lim_{\xi\to \partial \Omega} u^*(\xi)=0$ so by the generalized Walsh theorem (Proposition~3.2 in~\cite{Blocki_weak}) we get that $u^*\in \mathcal C(\bar\Omega)$. This means that $u=u^*$ and $u$ is a continuous function. Finally Theorem~\ref{thm_pmboundary} gives us that $u\in \mathcal {SH}_m(\bar \Omega)$. Note that $u$ is a continuous exhaustion function for $\Omega$.

\medskip

$(3) \Rightarrow (4):$ Let $z \in \partial \Omega$ and assume that there exists a function $\varphi \in \mathcal{SH}_m(\bar \Omega)$, $\varphi \neq 0$ such that $\varphi \leq 0$ and $\varphi(z)=0$. Let $\mu \in \mathcal{J}_z^m(\bar \Omega)$, then
\[
0=\varphi(z)\leq \int \varphi \, d\mu \leq 0\, .
\]
Hence $\operatorname{supp}(\mu) \subseteq \partial \Omega$.

\medskip

The proof of equivalence (1)$\Leftrightarrow$(5) is postponed to Corollary~\ref{smooth}.

\end{proof}

\section{An extension theorem}\label{sec_extention}

In this section we shall prove the extension theorem discussed in the introduction (Theorem~\ref{ext_in_pm_hyp}).
We provide also two new characterizations of $P_m$-hyperconvex domains (Corollary~\ref{cor2} and Theorem~\ref{localization_Pm}), and finally we prove that for a $P_m$-hyperconvex domain $\Omega$ one can find a continuous $m$-subharmonic exhaustion function on $\bar \Omega$, which is strictly $m$-subharmonic and smooth in $\Omega$ (Corollary~\ref{smooth}).

 We shall need the following lemma.

\begin{lemma}\label{lem}
  Assume that $\Omega$ is a $P_m$-hyperconvex domain in $\C^n$, $1 \leq m \leq n$,  and let $U$ be an open neighborhood of $\partial \Omega$. If $f\in\mathcal{SH}_m(U)\cap C^{\infty}(U)$ is a smooth function in some neighborhood of $\partial \Omega$, then there is a function $F \in \mathcal{SH}_m(\bar \Omega) \cap \mathcal{C}(\bar \Omega)$ such that $F=f$ on $\partial \Omega$.
\end{lemma}
\begin{proof}
Let $\psi\in \mathcal {SH}_m(\bar \Omega)\cap \mathcal C(\bar \Omega)$ be an exhaustion function for $\Omega$ (see Theorem~\ref{thm_charPmhx}). Let $U$ be an open set such that $\partial \Omega \subseteq U$ and $f\in \mathcal{SH}_m(U)\cap \mathcal C^{\infty}(U)$, and let $V$ be an open set such that $\partial \Omega \subseteq V \Subset U$. Moreover, let $K\subseteq \Omega$ be a compact set such that $\bar\Omega \subseteq K \cup U$ and $\partial K \subseteq V$.
Since $\Omega$ is also $m$-hyperconvex there exists a smooth and strictly $m$-subharmonic exhaustion function $\varphi$ for $\Omega$ (see~\cite{ACH}). Let $M>1$ be a constant large enough so that for all $z\in K$
\[
\varphi(z) - 1 > M \psi(z)\, .
\]
From Theorem~\ref{cont} part $(a)$ there exists an increasing sequence $\psi_j\in \mathcal{SH}_m(\Omega_j)\cap \mathcal \mathcal{C}(\bar \Omega_j)$, where $\bar \Omega \subseteq \Omega_j\Subset \Omega \cup V$, and such that $\psi_j\to \psi$ uniformly on $\bar \Omega$ so that
\[
\psi-\psi_j<\frac {1}{Mj}\, .
\]
Let us define
\[
\varphi_j :=
		\begin{cases}
			\max\left\{\varphi-\frac{1}{j}, M \psi_j\right\}, & \text{ if } z \in \Omega\, ,\\
			M\psi_j, & \text{ if } z \in \Omega_j \setminus \Omega\, .
		\end{cases}
\]
Note that the function $\varphi_j$ is $m$-subharmonic and continuous on $\Omega_j$, and $\varphi_j=\varphi-\frac{1}{j}$ on $K$. Next let $g$ be a smooth function such that $g=1$ on $V$, and $\operatorname{supp}(g)\subseteq U$. Since $\varphi_j$ is strictly $m$-subharmonic on the set where $g$ is non-constant, we can choose a constant $C$ so large that the function
\[
F_j:=C\varphi_j+gf
\]
belongs to $\mathcal{SH}_m^o(\bar\Omega)$. Observe that
\[
\max \{\varphi,M\psi\}-\max\left\{\varphi-\frac{1}{j}, M\psi_j\right\} \leq \frac{1}{j}\, ,
\]
and define
\[
F:=C\max\{\varphi,M\psi\}+gf.
\]
Then we have that
\[
F \geq  F_j \geq F-\frac{1}{j}\, ,
\]
and therefore, by uniform convergence, we get that $F\in \mathcal{SH}_m(\bar\Omega)\cap \mathcal{C}(\bar\Omega)$. Furthermore, for $z \in \partial \Omega$, we have that
\[
0 \leq f(z)-F_j(z)=-C\varphi_j(z)=-CM\psi_j(z)\to -CM\psi (z)=0\,\text{ as } j\to \infty\, ,
\]
and we see that $F=f$ on $\partial \Omega$.
\end{proof}

Now we state and prove the main theorem of this section.

\begin{theorem}\label{ext_in_pm_hyp}
Let $\Omega$ be a bounded $P_m$-hyperconvex domain in $\mathbb C^n$, $1 \leq m \leq n$, and let $f$ be a real-valued function defined on $\partial \Omega$. Then the following are equivalent:

\medskip

\begin{enumerate}\itemsep2mm

\item there exists $F\in \mathcal{SH}_m(\bar \Omega)$ such that $F=f$ on $\partial \Omega$;

\item $f\in \mathcal{SH}_m(\partial \Omega)$.
\end{enumerate}
Furthermore, if $f$ is continuous on $\partial \Omega$, then the function $F$ can be chosen to be continuous on $\Omega$.
\end{theorem}
\begin{proof} (1)$\Rightarrow$(2): Follows immediately from Corollary~\ref{cor}.

\bigskip

(2)$\Rightarrow$(1): Let $f\in \mathcal{SH}_m(\partial \Omega)$, then by Theorem~\ref{cont} part $(b)$ there exists a decreasing sequence $u_j\in\mathcal{SH}_m^o(\partial \Omega)$ of smooth functions such that $u_j\to f$, $j\to \infty$. By assumption $\Omega$ is in particular a regular domain, and therefore there is a sequence of harmonic functions $h_j$ defined on $\Omega$, continuous on $\bar \Omega$ such that $h_j=u_j$ on $\partial \Omega$. Define
\[
Sh_j=\sup\left \{v\in \mathcal{SH}_m(\bar \Omega): v\leq h_j\right\}\, ,
\]
then by Theorem~\ref{thm_edwards} we have that
\[
Sh_j=\sup\left \{v\in \mathcal{SH}_m(\bar \Omega)\cap\mathcal \mathcal{C}(\bar \Omega): v\leq h_j\right\}\, .
\]
Hence, $Sh_j$ is lower semicontinuous. Next we shall prove that in fact $Sh_j$ is continuous. By Lemma~\ref{lem} there exists $H_j\in \mathcal{SH}_m(\bar \Omega)\cap\mathcal C (\bar \Omega)$ such that $H_j=h_j$ on $\partial \Omega$. This implies that $H_j\leq Sh_j\leq (Sh_j)^*$, so $(Sh_j)^*=h_j=H_j$ on $\partial \Omega$. Note also that for all $z\in \partial \Omega$, and all $\mu\in \mathcal J_z^m(\bar \Omega)$ it holds that $\operatorname{supp}(\mu)\subseteq \partial \Omega$ by Theorem~\ref{thm_charPmhx} and then
\[
\int_{\bar \Omega}(Sh_j)^*\,d\mu=\int_{\partial \Omega}(Sh_j)^*\,d\mu=\int_{\partial \Omega}H_j\,d\mu=\int_{\bar \Omega}H_j\,d\mu\geq H_j(z)=(Sh_j)^*(z)\, ,
\]
and therefore by Theorem~\ref{thm_pmboundary} $(Sh_j)^*\in \mathcal{SH}_m(\bar \Omega)$, so $(Sh_j)^*=Sh_j$ and finally $Sh_j\in \mathcal{SH}_m(\bar \Omega)\cap\mathcal{C}(\bar \Omega)$. Now let
\[
F=\lim_{j\to \infty}Sh_j\, .
\]
Observe that $F=f$ on $\partial \Omega$, and $F\in \mathcal{SH}_m(\bar \Omega)$, since it is the limit of a decreasing sequence $Sh_j\in \mathcal{SH}_m(\bar \Omega)\cap\mathcal C(\bar \Omega)$.

\bigskip

To prove the last statement of this theorem assume that $f\in \mathcal C(\partial \Omega)$. Let $h$ be a harmonic function on $\Omega$ that is continuous on $\bar \Omega$ with boundary values $f$. As in the previous part of the proof define
\[
Sh=\sup\left \{v\in \mathcal{SH}_m(\bar \Omega): v\leq h\right\}=\sup\left \{v\in \mathcal{SH}_m(\bar \Omega)\cap\mathcal C(\bar \Omega): v\leq h\right\}\, ,
\]
so $Sh$ is lower semicontinuous. Furthermore, since $Sh\leq Sh_j$, then $(Sh)^*\leq (Sh_j)^*=Sh_j$ and
\[
(Sh)^*\leq \lim_{j\to \infty}Sh_j=F\leq Sh\, ,
\]
we have that $(Sh)^*=Sh=F\in \mathcal{SH}_m(\bar \Omega)\cap\mathcal C(\bar \Omega)$.
\end{proof}

Earlier we saw that if $\Omega$ is $P_m$-hyperconvex and if $z \in \partial \Omega$, then the measures in $\mathcal{J}_z^m(\bar \Omega)$ only have support on $\partial \Omega$. Following the line of \cite{HP} we will now see that, when $\Omega$ is $P_m$-hyperconvex, we actually have that $\mathcal J_z^m(\bar \Omega)=\mathcal J_z^m(\partial \Omega)$ for $z \in \partial \Omega$. This gives us another characterization of $P_m$-hyperconvex domains.

\begin{corollary}\label{cor2}
Let $\Omega$ be a bounded domain in $\C^n$. Then $\Omega$ is $P_m$-hyperconvex if, and only if, for all $z\in \partial \Omega$ we have $\mathcal J_z^m(\bar \Omega)=\mathcal J_z^m(\partial \Omega)$.
\end{corollary}
\begin{proof}
First assume that $\Omega$ is $P_m$-hyperconvex. It is clear that $\mathcal J_z^m(\partial \Omega)\subseteq\mathcal J_z^m(\bar \Omega)$. To prove the converse inclusion take $z\in \partial \Omega$, $\mu \in \mathcal J_z^m(\bar\Omega)$ and $f\in \mathcal{SH}_m^o(\partial \Omega)$, then $f\in \mathcal{SH}_m(\partial \Omega)\cap \mathcal C(\partial \Omega)$ and by Theorem~\ref{ext_in_pm_hyp} there exists $F\in \mathcal{SH}_m(\bar\Omega)$ such that $F=f$ on $\partial \Omega$. For $z\in \partial \Omega$ and $\mu\in \mathcal J_z^m(\bar \Omega)$ we have $\operatorname{supp}(\mu)\subseteq \partial \Omega$ and
\[
f(z)=F(z)\leq \int_{\bar\Omega}F\,d\mu=\int_{\partial \Omega}F\,d\mu=\int_{\partial\Omega}f\,d\mu,
\]
which means that $\mu\in \mathcal J_z^m(\partial \Omega)$.

For the converse implication assume that for all $z\in \partial \Omega$ we have $\mathcal J_z^m(\bar \Omega)=\mathcal J_z^m(\partial \Omega)$, then for all $z\in \partial \Omega$ and all  $\mu \in \mathcal J_z^m(\bar \Omega)$ we have $\operatorname{supp}(\mu)\subseteq \partial \Omega$ so by Theorem~\ref{thm_charPmhx} $\Omega$ is $P_m$-hyperconvex.

\end{proof}

On $P_m$-hyperconvex domains, we can now characterize the functions $u \in \mathcal{SH}_m(\bar \Omega)$ as those functions that are in $\mathcal{SH}_m(\Omega)$  and $u|_{\partial \Omega} \in \mathcal{SH}_m(\partial \Omega)$.

\begin{theorem}\label{cor_msubbdvalue}
Let $\Omega$ be a bounded $P_m$-hyperconvex domain in $\mathbb C^n$. Then $u\in \mathcal{SH}_m(\bar \Omega)$ if, and only if, $u\in \mathcal{SH}_m(\Omega)$, and $u\in \mathcal{SH}_m(\partial \Omega)$.
\end{theorem}
\begin{proof}
It follows from Theorem~\ref{thm_pmboundary} and Corollary~\ref{cor2}.
\end{proof}

As a corollary we obtain that for $P_m$-hyperconvex domains the exhaustion function can be chosen to be strictly $m$-subharmonic and  smooth, as it was announced in Theorem~\ref{thm_charPmhx}.

\begin{corollary}\label{smooth}
Let $\Omega$ be a bounded $P_m$-hyperconvex domain in $\mathbb C^n$. Then $\Omega$ admits a continuous negative exhaustion function which is $m$-subharmonic on $\bar \Omega$, smooth and strictly $m$-subharmonic on $\Omega$.
\end{corollary}
\begin{proof}
Since $\Omega$ is also $m$-hyperconvex then there exists a negative exhaustion function $\varphi\in \mathcal {SH}_m(\Omega)\cap \mathcal C^{\infty}(\Omega)$, which is strictly $m$-subharmonic on $\Omega$. Now it follows from Theorem~\ref{cor_msubbdvalue} that $\varphi \in \mathcal {SH}_m(\bar\Omega)$.
\end{proof}

Finally, we can prove that if a domain is locally $P_m$-hyperconvex then it is globally $P_m$-hyperconvex.

\begin{theorem}\label{localization_Pm}
Let $\Omega$ be a bounded domain in $\mathbb C^n$ such that for every $z\in \partial \Omega$ there exists a neighborhood $U_z$ such that $\Omega\cap U_z$ is $P_m$-hyperconvex, then $\Omega$ is $P_m$-hyperconvex.
\end{theorem}
\begin{proof} Assume that $\Omega$ is locally $P_m$-hyperconvex. Then it is also locally $m$-hyper\-con\-vex. By Theorem~3.3 in~\cite{ACH}, we know that $\Omega$ must be globally $m$-hyperconvex. Thus, there exists $\psi \in \mathcal{SH}_m(\Omega)\cap \mathcal{C}(\bar \Omega)$, $\psi \not\equiv 0,$ such that $\psi|_{\partial \Omega}=0$. We shall now show that $\psi \in \mathcal{SH}_m(\bar \Omega)$. Thanks to Theorem~\ref{localization} it is enough to show that for every $z \in \bar \Omega$ there is a ball $B_z$ such that $\psi|_{\bar \Omega \cap \bar B_z} \in \mathcal{SH}_m(\bar \Omega \cap \bar B_z)$.

For $z \in \Omega$ there exists $r>0$ such that $B(z,r)\Subset \Omega$ and then $\psi|_{B(z,r)}\in \mathcal {SH}_m(B(z,r))\cap \mathcal C(\bar B(z,r))$. Since,
\begin{equation}\label{ball}
\mathcal J^m_z(\bar B(z,r))=\mathcal J^m_z(\partial B(z,r))=\{\delta_z\}\, ,
 \end{equation}
we have that $\psi\in\mathcal {SH}_m(\partial B(z,r))$ and therefore by Corollary~\ref{cor_msubbdvalue} we have that $\psi\in\mathcal {SH}_m(\bar B(z,r))$.

Now it is sufficient to look at $z \in \partial \Omega$. Fix $z_0 \in \partial \Omega$, and a small ball $B_{z_0}$ around $z_0$. Without loss of generality assume that $B_{z_0} \Subset U_{z_0}$ such that $\Omega \cap B_{z_0}$ is $P_m$-hyperconvex. Once again, by Corollary \ref{cor_msubbdvalue} it is enough to show that $\psi \in \mathcal{SH}_m\bigl(\partial (\Omega \cap B_{z_0})\bigr)$, i.e. for every $z \in \partial (\Omega \cap B_{z_0})$, and every $\mu \in \mathcal{J}_{z}^m(\partial (\Omega \cap B_{z_0}))$, it holds
\begin{equation}\label{eq:locally}
\psi(z)\leq \int \psi \, d\mu\, .
\end{equation}
Suppose that $z \in \partial \Omega \cap B_{z_0}\setminus \partial B_{z_0}$. First we shall show that if $\mu \in \mathcal{J}_z^m(\partial(\Omega \cap B_{z_0}))$, then $\mu$ has support on $\partial \Omega$ and therefore condition $(\ref{eq:locally})$ will be fulfilled. Since $\Omega \cap U_{z_0}$ is $P_m$-hyperconvex, it has an exhaustion function $\varphi \in \mathcal{SH}_m(\bar \Omega \cap \bar {U}_{z_0})$,  and especially $\varphi \in \mathcal{SH}_m\left(\partial(\bar \Omega \cap \bar {B}_{z_0})\right)$. Let $\mu \in \mathcal{J}_z^m\left(\partial(\Omega \cap B_{z_0})\right)$, then we have
\[
0=\varphi(z)\leq \int \varphi \, d\mu \leq 0\, ,
\]
which means that $\mu$ has support where $\varphi=0$, i.e. on $\partial \Omega$.\\

Next, suppose that $z \in \bar\Omega \cap \partial B_{z_0}$. We claim that
\[
\mathcal{J}_{z}^m\left(\partial(\Omega \cap B_{z_0})\right)=\{\delta_z\}\, ,
\]
and this makes that $(\ref{eq:locally})$ holds. From (\ref{ball}) and from Theorem~\ref{ext_in_pm_hyp} we know that for every $z \in \bar \Omega \cap \partial B_{z_0}$ there exists a function $\varphi \in \mathcal {SH}_m(\bar B_{z_{0}})\subseteq\mathcal{SH}_m\left(\bar\Omega \cap \bar B_{z_0}\right)$ such that $\varphi(z)=0$ and $\varphi(\xi)<0$ for every $\xi \neq z$. By the same argument as above, we see that $\mathcal{J}_{z}^m\left(\partial(\Omega \cap B_{z_0})\right)=\{\delta_z\}$.
\end{proof}

\section{Some concluding remarks on approximation}\label{sec approx}

Approximation is a central part of analysis. The type of approximation needed depends obviously on the situation at hand.  In connection with Theorem~\ref{cont} one can ask the following question. Let $\Omega\subseteq\C^n$ be a bounded and open set, and let $1\leq m\leq n$. Under what assumptions on $\Omega$ do we have that
\begin{equation}\label{question}
\mathcal{SH}_m(\Omega) \cap \mathcal C(\bar \Omega)=\mathcal{SH}_m(\bar\Omega) \cap \mathcal C(\bar \Omega)?
\end{equation}

In the case when $m=1$, this type of theorem can be traced back to the work of Walsh~\cite{Walsh}, Keldysh~\cite{Keldych}, and
Deny~\cite{deny, deny2}, where they considered harmonic functions. In the harmonic case, some call this theorem the approximation theorem of Keldysh-Brelot after the contributions~\cite{Keldych,Brelot, Brelot2}. For subharmonic functions this type of approximation is included in the inspiring work of Bliedtner and Hansen~\cite{BlietnerHansen2} (see also~\cite{BlietnerHansen1,Hansen}). The articles mentioned are in a very general
setting. For us here it suffice to mention:

\begin{theorem}\label{thm_question}
Let $\Omega$ be a bounded domain in $\RE^n$. The following assertions are then equivalent:
\begin{enumerate}\itemsep2mm
\item for each $u$ in $\mathcal{SH}(\Omega) \cap \mathcal C(\bar \Omega)$ and each $\varepsilon>0$ there is a function $v$ in $\mathcal{SH}(\bar\Omega) \cap \mathcal C(\bar \Omega)$ such that $|u-v|<\varepsilon$ on $\bar\Omega$;

\item the sets $\RE^n\backslash\bar{\Omega}$, and $\RE^n\backslash \Omega$, are thin at the same points of $\bar\Omega$.
\end{enumerate}
\end{theorem}

\noindent For further information on the case $m=1$ we refer to the inspiring survey written by Hedberg~\cite{Hedberg} (see also~\cite{gardiner}).

If we look at the other end case of the Caffarelli-Nirenberg-Spruck model, when $m=n$, and we are in the world of pluripotential, then our approximation question bear resemblance with the so called Mergelyan approximation of holomorphic
function. Therefore, some call~(\ref{question}) the $\mathcal{PSH}$-Mergelyan property (see e.g.~\cite{H}).
The first positive result for the $\mathcal{PSH}$-Mergelyan property is due to Sibony. In 1987, he proved in~\cite{S} that every smoothly bounded and pseudoconvex domain has this property. Later Forn\ae ss and Wiegerinck~\cite{FW} generalized this in their beautiful paper to arbitrary domains with $C^1$-boundary. Recently, Persson and Wiegerinck~\cite{PW} proved that a domain of which the boundary is continuous with the possible exception of a countable set of boundary points, has the $\mathcal{PSH}$-Mergelyan property (this generalize~\cite{avelin2,H}). Furthermore, in~\cite{PW} they constructed very enlightening examples that show  that there can be no corresponding Theorem~\ref{thm_question} in the case $m=n$.

 \bigskip

 At this point there is no satisfactory answer to question~(\ref{question}) within the Caffarelli-Nirenberg-Spruck framework that covers
  the knowledge of the end cases $m=1$, and $m=n$. Even so, in Theorem~\ref{bm-reg} we give a family of bounded domains that satisfies~(\ref{question}), and we prove several characterizations of this type of domains. Obviously, there are domains that satisfies~(\ref{question}), and is not included in Theorem~\ref{bm-reg}. For further information, and inspiration, on  approximation we refer to~\cite{gardiner,Gauthier2} and the references therein.

\begin{theorem} \label{bm-reg}
Assume that $\Omega$ is a bounded domain in $\C^n$, $n\geq 2$, $1\leq m\leq n$. Then the following assertions are equivalent:

\medskip

 \begin{enumerate}\itemsep2mm

\item  for every continuous function $f:\partial\Omega\to \RE$  we have that
\[
\operatorname{PB}^m_f\in\mathcal{SH}_m(\bar{\Omega})\cap \mathcal C(\bar {\Omega})\, ,
\text{ and } \operatorname{PB}^m_f=f \text{ on } \partial \Omega\, ,
\]
where
\[
\operatorname{PB}^m_f(z)=\sup\Bigg\{v(z): v\in\mathcal{SH}_m(\bar \Omega),\; v(\xi)\leq f(\xi)\, , \;\; \forall \xi\in\partial\Omega\Bigg\}\, ;
\]

\item $\partial\Omega$ has a strong barrier at every point $z_0\in\partial\Omega$ that is $m$-subharmonic on $\bar{\Omega}$, i.e. there exists a $m$-subharmonic function $u:\bar{\Omega}\to\RE$ such that
     \[
     \lim_{x\to y_0\atop x\in\Omega} u(x)=0\, ,
     \]
     and
     \[
      \limsup_{x\to y\atop x\in\Omega} u(x)<0 \qquad  \text{ for all }  y\in\bar{\Omega}\backslash\{y_0\}\, ;
     \]

\item $\Omega$ admits an exhaustion function $\varphi$ that is negative, continuous, $m$-subharmonic on  $\bar{\Omega}$, smooth on $\Omega$, and such that
    \[
   \left(\varphi(z)-|z|^2\right)\in \mathcal{SH}_m(\bar{\Omega})\, ;
    \]

 \item  for every $z\in \partial \Omega$ we have that $\mathcal{J}_z^m(\bar \Omega)=\{\delta_z\}$.

  \end{enumerate}
 \end{theorem}
\begin{proof}

$(1)\Rightarrow(4):$ Fix $z\in \partial \Omega$, $\mu\in \mathcal J_z^m(\bar \Omega)$, and let $f$ be a real-valued continuous function defined on $\partial \Omega$ such that $f(z)=0$ and $f(\xi)<0$ for $\xi \neq z$. Then it holds that
\[
0=\operatorname{PB}^m_f(z)\leq \int\operatorname{PB}^m_f\,d\mu\leq 0\, ,
\]
and therefore it follows that $\operatorname{supp}(\mu)\subseteq \{z\}$, Thus, $\mu=\delta_z$.

\medskip

$(4)\Rightarrow(1):$  First note that it follows from (4) that every continuous functions defined on the boundary $\partial\Omega$ is $m$-subharmonic on $\partial\Omega$ in the sense of Definition~\ref{def_msubkomp}. Let $f\in \mathcal C(\partial \Omega)$, then by Theorem~\ref{ext_in_pm_hyp} and Theorem~\ref{thm_charPmhx} there exists a function $F\in \mathcal {SH}_m(\bar \Omega)\cap \mathcal C(\bar \Omega)$ such that $F=f$ on $\partial \Omega$. Let us define
\[
\textbf{S}_f(z)=\sup\Bigg\{v(z): v\in\mathcal{SH}_m(\Omega),\; \varlimsup_{\zeta\rightarrow\xi \atop \zeta\in\Omega}v(\zeta)\leq f(\xi)\, , \;\; \forall \xi\in\partial\Omega\Bigg\}\, .
\]
For $z\in \Omega$, we have that
\[
F(z)\leq \operatorname{PB}^m_f(z)\leq \textbf{S}_f(z)\, ,
\]
and then
\[
\lim_{\zeta\to \xi}\textbf{S}_f(\zeta)=f(\xi)\,\qquad \text{for all } \xi\in \partial \Omega\, .
\]
Thanks to the the generalized Walsh theorem (Proposition~3.2 in~\cite{Blocki_weak}) we get that $\textbf{S}_f\in \mathcal {SH}_m(\Omega)\cap \mathcal C(\bar \Omega)$, and by Theorem~\ref{cor_msubbdvalue},  $\textbf{S}_f\in \mathcal {SH}_m(\bar \Omega)$. Hence,
$\operatorname{PB}^m_f=\textbf{S}_f$ and the proof is finished.

\medskip

$(1)\Rightarrow(2):$ Fix $z\in \partial \Omega$, and let $f$ be a continuous function on $\partial \Omega$ such that $f(z)=0$ and $f(\xi)<0$ for $\xi \neq z$. Then the function $\operatorname{PB}^m_f$ is a strong barrier at $z$.

\medskip

$(2)\Rightarrow(4):$ Fix $z\in \partial \Omega$, $\mu\in \mathcal J_z^m(\bar \Omega)$, and let $u_z$ be a strong barrier at $z$. Then it holds
that
\[
0=u_z(z)\leq \int u_z\,d\mu\leq 0\, ,
\]
and therefore $\operatorname{supp}(\mu)\subseteq \{z\}$. Thus, $\mu=\delta_z$.

\medskip

$(4)\Rightarrow(3):$ Let $\mathcal J_z^{c,m}$ be the class of Jensen measures defined by continuous $m$-subharmonic functions on $\Omega$ (see~\cite{ACH}), i.e. $\mu\in \mathcal J_z^{c,m}$ if
 \[
  u(z)\leq \int_{\bar{\Omega}} u \, d\mu\, , \text{ for all } u  \in \mathcal{SH}_m(\Omega) \cap \mathcal C(\bar \Omega)\, .
  \]
Let $z\in \partial \Omega$, and note that $\mathcal {SH}_m^o(\bar \Omega)\subseteq \mathcal{SH}_m(\Omega) \cap \mathcal C(\bar \Omega)$, so
\[
\mathcal J_z^{c,m}\subseteq \mathcal J_z^m(\bar \Omega)=\{\delta_z\}\, .
\]
Therefore by Theorem~4.3 in~\cite{ACH}, there exits an exhaustion function $\varphi$ that is negative, smooth,  $m$-subharmonic on $\Omega$, continuous on $\bar \Omega$, and such that
    \[
   \left(\varphi(z)-|z|^2\right)\in \mathcal{SH}_m(\Omega)\, .
    \]
Condition (4) implies that every continuous function defined on the boundary is also $m$-subharmonic. This means that $(\varphi(z)-|z|^2)\in \mathcal{SH}_m(\partial\Omega)$. Finally, Theorem~\ref{cor_msubbdvalue} gives  us that $(\varphi(z)-|z|^2)\in \mathcal {SH}_m(\bar \Omega)$.

\medskip

$(3)\Rightarrow(2):$ Condition (3) implies that for $z\in\partial\Omega$ we have that
\[
\varphi(z)-|z|^2=-|z|^2\in \mathcal {SH}_m(\partial \Omega)\, .
\]
Take $z_0\in \partial \Omega$, and note that
\[
-|z-z_0|^2=-|z|^2+z\bar z_0+\bar zz_0-|z_0|^2\in \mathcal {SH}_m(\partial \Omega)\, .
\]
Theorem~\ref{ext_in_pm_hyp} and Theorem ~\ref{thm_charPmhx} imply that there exists $F\in \mathcal {SH}_m(\bar \Omega)\cap \mathcal C(\bar \Omega)$ such that $F=-|z-z_0|^2$ on $\partial \Omega$. The function $F$ is a strong barrier at $z_0$.
\end{proof}

\end{document}